\documentclass[11pt]{amsart}
\usepackage{graphicx}
\usepackage[mathscr]{eucal}
\usepackage{microtype}
\usepackage{amssymb}
\usepackage[parfill]{parskip}
\usepackage{mathrsfs}
\usepackage{xcolor}
\usepackage{color}
\usepackage[top=1in, bottom=1in, left=1in, right=1in]{geometry}
\usepackage{enumerate}
\usepackage{bm}
\usepackage[normalem]{ulem}
\usepackage[all]{xy}

\newcommand{\rt}{\mathnormal{\mathsf{RT}}}
\newcommand{\hind}{\mathnormal{\mathsf{HT}}}
\newcommand{\zfa}{\mathnormal{\mathsf{ZFA}}}
\newcommand{\zf}{\mathnormal{\mathsf{ZF}}}
\newcommand{\zfc}{\mathnormal{\mathsf{ZFC}}}
\newcommand{\ac}{\mathnormal{\mathsf{AC}}}
\newcommand{\fdf}{\mathnormal{\mathrm{Fin}=\mathrm{D}\text{-}\mathrm{Fin}}}
\newcommand{\dc}{\mathnormal{\mathsf{DC}}}
\newcommand{\bpi}{\mathnormal{\mathsf{BPI}}}
\newcommand{\oep}{\mathnormal{\mathsf{OEP}}}
\newcommand{\op}{\mathnormal{\mathsf{OP}}}
\newcommand{\kl}{\mathnormal{\mathsf{KL}}}
\newcommand{\kw}{\mathnormal{\mathsf{KW}}}
\newcommand{\cc}{\mathnormal{\mathsf{CC}}}
\DeclareMathOperator{\fu}{FU}
\DeclareMathOperator{\fs}{FS}

\DeclareMathOperator{\dom}{dom}
\DeclareMathOperator{\ran}{ran}

\DeclareMathOperator{\sym}{Sym}

\newtheorem{theorem}{Theorem}[section]
\newtheorem{proposition}{Proposition}[section]
\newtheorem{lemma}{Lemma}[section]
\newtheorem{corollary}{Corollary}[section]

\newtheorem{question}{Question}[section]

\theoremstyle{definition}
\newtheorem{definition}{Definition}[section]

\theoremstyle{remark}

\author[D. Fern\'andez]{David Fern\'andez-Bret\'on}
\address{
Escuela Superior de F\'{\i}sica y Matem\'aticas\\
Instituto Polit\'ecnico Nacional\\
Av. Instituto Polit\'ecnico Nacional s/n Edificio 9, 
Col. San Pedro Zacatenco, Alcald\'{\i}a Gustavo A. Madero, 07738, CDMX, Mexico. 
}
\email{dfernandezb@ipn.mx}
\urladdr{https://dfernandezb.web.app}

\title[Hindman's Theorem and Weak Choice]{Hindman's theorem in the \\ hierarchy of choice principles}

\begin{document}

\maketitle

\begin{abstract}
In the context of $\zf$, we analyze a version of Hindman's finite unions theorem on infinite sets, which normally requires the Axiom of Choice to be proved. We establish the implication relations between this statement and various classical weak choice principles, thus precisely locating the strength of the statement as a weak form of the $\ac$.
\end{abstract}

\section{Introduction}

One of the central results of infinitary (countable) Ramsey theory is the so-called Hindman's finite sums theorem~\cite{hindman-thm}, stating that for every finite partition of $\mathbb N$ it is possible to find elements $x_1<\cdots<x_n<\cdots$ such that all sums of finitely many of the $x_i$, with no repetitions, are contained in the same cell of the partition. An extremely close result in a similar vein, which was in fact already known to be equivalent to Hindman's finite sums theorem before the latter was proved, is the statement that for every partition of the set $[\mathbb N]^{<\omega}$ of all finite subsets of $\mathbb N$, one can find infinitely many pairwise disjoint sets such that all unions of finitely many of them are contained within the same cell of the partition. Upon replacing $\mathbb N$ with an arbitrary set $X$ in the latter result, one obtains a statement that, while provable in $\zfc$, may potentially not be a theorem of $\zf$. This statement is what we will refer to as {\it Hindman's theorem} in this paper, and it will be our central object of study.

\begin{definition}\label{def:ht}
{\em Hindman's theorem}, denoted $\hind$, is the statement that, for every infinite set $X$ and for every colouring $c:[X]^{<\omega}\longrightarrow 2$ of the finite powerset of $X$ with two colours, there exists an infinite, pairwise disjoint family $Y\subseteq[X]^{<\omega}$ such that the set
\begin{equation*}
\fu(Y)=\left\{\bigcup_{y\in F}y\bigg| F\in[Y]^{<\omega}\setminus\{\varnothing\}\right\}
\end{equation*}
is $c$-monochromatic.
\end{definition}

(We prove later, in Proposition~\ref{colourblind-hindman}, that we obtain an equivalent statement, modulo $\zf$, if we vary the number of colours in the colouring, so long as said number remains finite.) It follows from Hindman's finite unions theorem over $\mathbb N$ that $\hind$ is a theorem of $\zfc$ (by simply embedding $\mathbb N$ into any infinite set $X$ and restricting any colouring of $[X]^{<\omega}$); however, it turns out that one cannot prove $\hind$ in $\zf$ only. Hence, one can think of the statement $\hind$ as a weak form of the Axiom of Choice, and it thus makes sense to try and compare this choice principle with other classical choice principles that have been extensively studied, investigating the implication relations (modulo $\zf$) that there are between them. It is worth noting that $\hind$ is a very natural choice principle not only due to its origins in Ramsey theory, but also in light of some results presented in this paper, e.g., Proposition~\ref{konigplushindman}, stating that the conjunction of $\hind$ and K\"onig's Lemma is equivalent to the statement that every infinite set is Dedekind-infinite (and therefore, the latter is also equivalent to the conjunction of $\hind$ and Ramsey's theorem). It is also worth noting that $\hind$ is equivalent to a statement that simply deals with the Dedekind-finiteness of finite powersets of sets (see Proposition~\ref{equiv-hindman} below), suggesting that the algebraic and topological dynamics aspects of Hindman's Theorem are not particularly relevant from the point of view of choice principles.

\begin{figure}
\begin{center}
\begin{tabular}{c}
\xymatrix{
 & & \ac\ar@{=>}[dl]\ar@{=>}[ddr]\ar@{=>}[d] & \\
 & \dc\ar@{=>}[d] & \bpi\ar@{=>}[dd] & \\
 & \cc\ar@{=>}[d] & & \kw\ar@{=>}[dd] \\
 & \fdf\ar@{=>}[dl]\ar@{=>}[dr] & \oep\ar@{=>}[dr] & \\
\hind\ar@{=>}[dr] & & \rt\ar@{=>}[d] & \op\ar@{=>}[dl] \\
 & \text{Form 82}\ar@{=>}[d] & \kl\ar@{=>}[dl] & \\
 & \varnothing & & 
}
\end{tabular}
\end{center}
\caption{Implications between $\hind$ and other classical choice principles.}
\label{fig:diagram}
\end{figure}
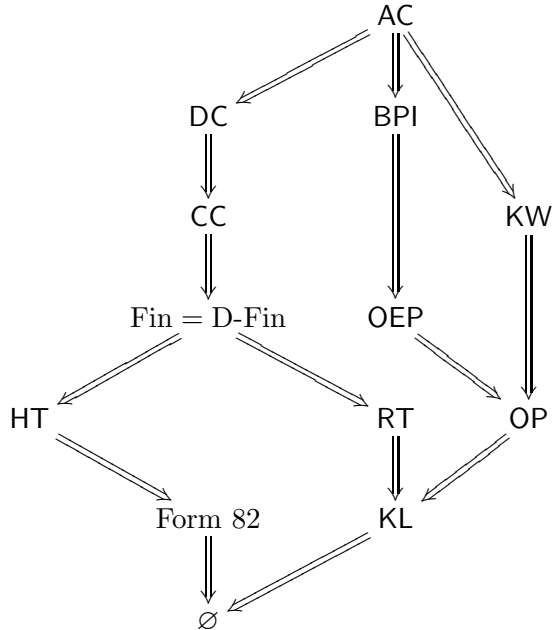

In this paper, we locate the precise strength of $\hind$ among the most important classical choice principles\footnote{Further work of E. Tachtsis~\cite{tachtsis-hindman} has established some equivalences between $\hind$ and other previously known weak choice principles.}. The choice principles considered are, in addition to the Axiom of Choice, the principle of Countable Choice, the axiom of Dependent Choice, K\"onig's Lemma, the principle that every Dedekind-finite set is finite, the Boolean Prime Ideal theorem, the Kinna--Wagner selection principle, the Ordering Principle, the Order Extension Principle, Ramsey's Theorem, and Form 82 from~\cite{howard-rubin} (the latter is not as classical as the other ones, but we include it in our study due to its high degree of similarity with a certain equivalence of $\hind$). The results we obtain are summarized in the
diagram from Fig.~\ref{fig:diagram}, which contains all possible ZF-provable implications between $\hind$ and the aforementioned choice principles, each of which is represented by the obvious abbreviation in the diagram. Formal definitions of each of the choice principles considered are to be found in Sec.~\ref{zf-results}.

The reader will note that the diagram from Fig.~\ref{fig:diagram} contains very few implication arrows to and from $\hind$ (the only ones are $\fdf\Rightarrow\hind$ and $\hind\Rightarrow\text{Form 82}$, both of which will be obvious, given the equivalence of $\hind$ established in Proposition~\ref{equiv-hindman}, once we state the meaning of the involved choice principles). Therefore, we must emphasize that the main body of work presented in this paper is not proofs of implications in $\zf$, but rather {\it independence proofs}, showing that there are no further implications between $\hind$ and any other of the principles mentioned. In other words, the most meaningful information that can be gathered from the diagram in Fig.~\ref{fig:diagram} is not the arrows shown, but rather the ones {\it not} shown, signalling that an independence proof (or an argument stemming from a previous independence proof) has been established formally. As such, most of the content of this paper deals either with symmetric models, or with Fraenkel--Mostowkski permutation models (which yield symmetric models after applying well-known transfer theorems), thus obtaining models of $\zf$ witnessing the unprovability of the relevant statement. In Sec.~\ref{zf-results} we relay a few basic $\zf$ results; afterwards in Sec.~\ref{consistency-results} we discuss Fraenkel--Mostowski permutation models and the transfer theorems that allow us to obtain $\zf$ models from them, and proceed to determine whether $\hind$ holds in various of these models. With this information in hand, we establish all the independence results required to complete the diagram in Fig.~\ref{fig:diagram}, except for the reversibility of the implication $\hind\Rightarrow\text{Form 82}$; this is addressed in Sec.~\ref{ht-vs-form82}, which contains the proof that this implication is not reversible (and the proof is involved enough that it warrants its own section). Finally, in Sec.~\ref{section-boolean}, we consider a weaker Boolean version of $\hind$ and also determine its place within the hierarchy of choice principles (see the enhanced diagram from Fig.~\ref{fig:enhanced-diagram}).

\section{Some basic results}\label{zf-results}

We begin by establishing that, in our definition of $\hind$, we could have considered colourings on any finite number of colours and still obtained an equivalent statement modulo $\zf$. Therefore, we will temporarily use the symbol $\hind(k)$ (where $k\in\mathbb N\setminus\{1\}$) to denote that, for every infinite set $X$ and every colouring $c:[X]^{<\omega}\longrightarrow k$, there exists an infinite, pairwise disjoint $Y\subseteq[X]^{<\omega}$ such that $\fu(Y)$ is $c$-monochromatic. Hence, $\hind(2)$ is exactly what we called $\hind$ in Definition~\ref{def:ht}; after the following proposition, we will be able to drop the parameter $k$ and simply write $\hind$ in all cases.

\begin{proposition}\label{colourblind-hindman}
All of the statements $\hind(k)$, as $k\in\mathbb N\setminus\{1\}$ varies, are equivalent under $\zf$.
\end{proposition}

\begin{proof}
Since $k$-colourings are always also $k'$-colourings whenever $k\leq k'$, we have that $\hind(k')\Rightarrow\hind(k)$ under these circumstances. Now to finish the proof, we need only show that $\hind(k)\Rightarrow\hind(k+1)$ for $k\geq 2$ (which yields an argument by induction). So suppose that $k\geq 2$ and that $\hind(k)$ holds. Let $X$ be an infinite set and let $c:[X]^{<\omega}\longrightarrow k+1$ be a colouring. Define another colouring $d:[X]^{<\omega}\longrightarrow k$ by letting $d(x)=\min\{c(x),k-1\}$. Using $\hind(k)$ we obtain an infinite pairwise disjoint family $Y\subseteq [X]^{<\omega}$ such that $\fu(Y)$ is $d$-monochromatic, say on colour $i<k$. If $i<k-1$ then $\fu(Y)$ is $c$-monochromatic as well (on the same colour) and we are done; otherwise we know that for every $y\in\fu(Y)$, $c(y)\in\{k-1,k\}$. Hence, we can define yet another colouring $e:[Y]^{<\omega}\longrightarrow 2$ given by $e(F)=k-c\left(\bigcup_{y\in F}y\right)$ and use $\hind(2)$ to obtain an infinite, pairwise disjoint family $W\subseteq[Y]^{<\omega}$ such that $\fu(W)$ is $e$-monochromatic, say in colour $j<2$. This means that, for every $\mathcal F\in[W]^{<\omega}$, 
\begin{equation*}
j=e\left(\bigcup_{F\in\mathcal F}F\right)=k-c\left(\bigcup_{y\in\bigcup_{F\in\mathcal F}F}y\right),
\end{equation*}
so that, if we define
\begin{equation*}
Z=\left\{\bigcup_{y\in F}y\bigg|F\in W\right\},
\end{equation*}
then $Z\subseteq\fu(Y)\subseteq[X]^{<\omega}$ is an infinite, pairwise disjoint family such that $\fu(Z)$ is monochromatic for $c$ (in colour $k-j$), and we are done.
\end{proof}

The ``classical'' choice principles considered in this paper, in addition to $\hind$, are the following:

\begin{enumerate}
\item The {\em Axiom of Dependent Choice}, abbreviated $\dc$, is the statement that, for every set $X$ equipped with a relation $R\subseteq X\times X$ such that $(\forall x\in X)(\exists y\in X)(x\mathrel{R}y)$, there exists a countable sequence $\langle x_n\big|n<\omega\rangle$ such that $(\forall n<\omega)(x_n\mathrel{R}x_{n+1})$ (this statement is labelled Form 43 in~\cite{howard-rubin}).
\item The {\em Axiom of Countable Choice}, which we will abbreviate $\cc$, is the statement that every countable family of nonempty sets admits a choice function (Form 8 from~\cite{howard-rubin}).
\item The statement ``every infinite set is Dedekind-infinite'' will be denoted by $\fdf$ (Form 9 in~\cite{howard-rubin}).
\item {\em Ramsey's theorem}, denoted by $\rt$, is the statement that for every infinite set $X$ and for every colouring $c:[X]^2\longrightarrow 2$, there exists an infinite $Y\subseteq X$ such that $[Y]^2$ is $c$-monochromatic (Form 17 from~\cite{howard-rubin}).
\item {\em K\"onig's lemma}, which we will abbreviate $\kl$, is the statement that every countable family of nonempty finite sets admits a choice function (Form 10 in~\cite{howard-rubin}).
\item The {\em Boolean Prime Ideal theorem}, denoted by $\bpi$, is the statement that every Boolean algebra carries a prime ideal (Form 14 from~\cite{howard-rubin}).
\item The {\em Kinna--Wagner selection principle}, which will be abbreviated by $\kw$, is the statement that\footnote{Equivalently, for every family of sets $\mathscr F$ all of which have at least two elements, there is a function $f$ with domain $\mathscr F$ such that $(\forall S\in\mathscr F)(\varnothing\neq f(S)\subsetneq S)$, see~\cite[Problem 4.12]{jech-choice}. Sufficiently old papers refer to the Kinna--Wagner selection principle simply as the {\em selection principle}.} for every set $X$ there exists an ordinal number $\alpha$ and an injective function $f:X\longrightarrow\wp(\alpha)$ (Form 15 in~\cite{howard-rubin}).
\item The {\em Ordering Principle}, denoted by $\op$, is the statement that every set can be linearly ordered (Form 30 from~\cite{howard-rubin}).
\item The {\em Order Extension Principle}, abbreviated $\oep$, is the statement that every partial order can be extended to a linear order on the same set (Form 49 in~\cite{howard-rubin}).
\item Form 82 (according to the numbering in~\cite{howard-rubin}) is the statement that for every infinite set $X$, its powerset $\wp(X)$ is Dedekind-infinite.
\end{enumerate}

Recall that a set is said to be {\em Dedekind-infinite} if $\omega$ injects into it (equivalently, if there exists an injective, but not surjective, function of the set into itself), and a set is {\em Dedekind-finite} if it is not Dedekind-infinite. It is hard not to see that every Dedekind-infinite set must be infinite; however, the converse to this statement is not provable in $\zf$, and is therefore considered a choice principle. K\"onig's Lemma owes its name to the fact that it is equivalent, over $\zf$, to the classical theorem about finitely branching infinite trees due to K\"onig (that is, the statement that every finitely branching infinite tree must have an infinite branch). Well-known classical results in choiceless set theory establish that, over $\zf$, $\dc$ implies $\cc$ which in turn implies $\fdf$; $\bpi$ implies $\oep$, and either $\oep$ or $\kw$ implies $\op$ which in turn implies $\kl$. %, with no provable implication relation between $\bpi$ and $\kw$ or between $\oep$ and $\kw$. 
Moreover, $\fdf$ implies both Form 82 and $\rt$, and the latter in turn implies $\kl$. Furthermore, none of the implications mentioned in this paragraph is reversible, and there are no further implication relations between any of the choice principles mentioned in this paragraph (see, e.g.,~\cite{howard-rubin} for a complete set of references on all the facts just mentioned).

We now begin to analyse the strength of $\hind$ among all of these principles. Important information can be gathered by ``locally'' studying  those sets for which Hindman's finite unions theorem holds (as opposed to the ``global'' principle that every infinite set satisfies Hindman's theorem). Such a careful study was performed in~\cite{brot-cao-fernandez}, where the following definition is stated. 

\begin{definition}[\cite{brot-cao-fernandez}, Definition 3.6 (3), cf. Definition 3.1]
A set $X$ will be called {\em H-finite} if there exists a colouring $c:[X]^{<\omega}\longrightarrow 2$ such that for no infinite, pairwise disjoint $Y\subseteq[X]^{<\omega}$ can the set $\fu(Y)$ be $c$-monochromatic. We will say that $X$ is {\em H-infinite} if it is not H-finite (so $X$ is H-infinite if and only if Hindman's finite unions theorem holds at $X$).
\end{definition}

Thus, $\hind$ is simply the statement that every infinite set must be H-infinite. Hence, it follows from, e.g.,~\cite[Proposition 4.2]{brot-cao-fernandez} that $\hind$ is not provable in $\zf$ alone. We can get much more precise information after establishing the following equivalence of $\hind$.

\begin{proposition}\label{equiv-hindman}
In $\zf$, the statement $\hind$ is equivalent to the statement that for every infinite set $X$, its finite powerset $[X]^{<\omega}$ is Dedekind-infinite.
\end{proposition}

\begin{proof}
By~\cite[Theorem 3.2]{brot-cao-fernandez}, a set $X$ is H-finite if and only if $[X]^{<\omega}$ is Dedekind-finite. Thus the proposition follows immediately.
\end{proof}

\begin{corollary}\label{arrows-ht}
In $\zf$, $\fdf$ implies $\hind$, which in turn implies Form 82.
\end{corollary}

\begin{proof}
Immediate from Proposition~\ref{equiv-hindman}.
\end{proof}

In particular, by taking any model of $\zf$ in which $\fdf$ holds but $\ac$ fails, we see that $\hind$ is strictly weaker than the full Axiom of Choice. The fact that the implication $\fdf\Rightarrow\hind$ is not reversible is established in Sec.~\ref{consistency-results}. That the implication $\hind\Rightarrow\text{Form 82}$ is not reversible is the content of Sec.~\ref{ht-vs-form82}.

In light of Proposition~\ref{equiv-hindman}, we see that $\hind$ is precisely the piece that is missing from either $\kl$ or $\rt$ to get $\fdf$, as shown by the following proposition.

\begin{proposition}\label{konigplushindman}
In $\zf$, the following are equivalent:
\begin{enumerate}
\item $\fdf$,
\item $\rt\wedge\hind$,
\item $\kl\wedge\hind$.
\end{enumerate}
\end{proposition}

\begin{proof}\hfill
\begin{description}
\item[(1)$\Rightarrow$(2)] This is immediate from Corollary~\ref{arrows-ht} together with the well-known fact that $\fdf\Rightarrow\rt$.
\item[(2)$\Rightarrow$(3)] Immediate from the fact that $\rt\Rightarrow\kl$.
\item[(3)$\Rightarrow$(1)] Assume that $\hind$ and $\kl$ both hold, and let $X$ be an arbitrary infinite set. By Proposition~\ref{equiv-hindman}, $\hind$ implies that $[X]^{<\omega}$ is Dedekind-infinite and so there is a countable injective sequence $\langle F_n\big|n<\omega\rangle$ of finite subsets of $X$. Recursively replacing, if necessary, each $F_n$ with $F_m\setminus\left(\bigcup_{k<m}F_k\right)$, where $m\geq n$ is the least index such that this set is nonempty, we may assume that the $F_n$ are pairwise disjoint and nonempty. The sequence of $F_n$ forms a countable family of nonempty finite sets, so by K\"onig's lemma there is a choice function $f:\omega\longrightarrow\bigcup_{n<\omega}F_n\subseteq X$. Since the $F_n$ are pairwise disjoint and each $f(n)\in F_n$, we conclude that the function $f:\omega\longrightarrow X$ is in fact injective, and so $X$ is Dedekind-infinite.
\end{description}
\end{proof}

We finish the section with a couple more $\zf$ results that will be useful in the next section. To state the first one, we recall a definition from~\cite{herrlich-finiteinfinite}.

\begin{definition}[\cite{herrlich-finiteinfinite}, Definition 8]
A set $X$ is said to be {\em C-finite} if there is no surjection $f:X\longrightarrow\omega$, and it is {\em C-infinite} if it is not C-finite.
\end{definition}

C-finite sets were called {\em dually Dedekind-finite} by Degen~\cite{degen}. It follows from~\cite[Lemma 4.11]{herrlich-choice} that any set $X$ is C-finite if and only if $\wp(X)$ is Dedekind-finite. In particular, Form 82 can be thought of as the statement that every infinite set is C-infinite.

Recall also that a set is {\em amorphous} if it is infinite and its only subsets are the finite ones and the cofinite ones.

\begin{proposition}\label{amorphous-powerset}
In $\zf$, if $X$ is amorphous and C-infinite, then $X$ is H-infinite.
\end{proposition}

\begin{proof}
Since $\wp(X)$ is Dedekind-infinite, there is an injective sequence $\langle A_n\big|n<\omega\rangle$ of subsets of $X$. Now, since $X$ is amorphous, each $A_n$ is either a finite, or a cofinite, subset of $X$; using the pigeonhole principle, thin out the sequence by eliminating terms so that either all of the $A_n$ are finite, or all of the $A_n$ are cofinite. In the first case, let $F_n=A_n$; in the second case let $F_n=X\setminus A_n$, for all $n<\omega$. In either case, the sequence $\langle F_n\big|n<\omega\rangle$ is an injective sequence of elements of $[X]^{<\omega}$, and we are done.
\end{proof}

The next proposition, which is the last of the section, will be useful when determining whether $\hind$ holds in Cohen's model for the failure of the $\ac$.

\begin{proposition}\label{linearly-ordered-plus-hindman}
In $\zf$, if $X$ is a linearly orderable H-infinite set, then $X$ is Dedekind-infinite
\end{proposition}

\begin{proof}
Let $\leq$ be a linear order on $X$ and, since $X$ is H-infinite, let $\langle F_n\big|n<\omega\rangle$ be an injective sequence of finite subsets of $X$. Using the same trick as in the proof of Theorem~\ref{konigplushindman}, we may assume that the $F_n$ are pairwise disjoint. Hence, if we define $x_n=\min_{\leq}F_n$, the sequence $\langle x_n\big|n<\omega\rangle$ of elements of $X$ is injective. Therefore, $X$ is Dedekind-infinite.
\end{proof}

\section{Models of $\zf$ and $\zfa$}\label{consistency-results}

There are two main techniques for independence proofs that we use throughout this paper. The first one is by means of the forcing technique, passing to a special submodel of a forcing extension to get a model of $\zf$; models obtained in this way are called {\em symmetric models}. The only model arising from this technique that we will study in detail is Cohen's basic model, as described in~\cite[Sec. 5.3]{jech-choice}; this model is denoted $\mathcal M_1$ in~\cite{howard-rubin}. The other technique that will be used is that of the Fraenkel--Mostowski permutation models of $\zfa$, as described in~\cite[Secs. 4.1 and 4.2]{jech-choice}. The three ``classical'' Fraenkel--Mostowski models that we will study in this section are the First and Second Fraenkel Model (denoted by $\mathcal N_1$ and $\mathcal N_2$, respectively, in~\cite{howard-rubin}), and Mostowski's Linearly Ordered Model ($\mathcal N_3$ in~\cite{howard-rubin}). These models are described (each on a different section) in~\cite[Secs. 4.3--4.5]{jech-choice}, and any unexplained notation is used as in that source. In Sec.~\ref{ht-vs-form82}, we will build a new permutation model in order to show that Form 82 does not imply $\hind$.

\subsection{Transferable statements and finiteness classes}

Since we are ultimately interested in proofs of independence from $\zf$, rather than from $\zfa$, it is necessary to justify that independence proofs from the latter can be transferred to independence proofs from the former, for statements like the ones we will consider in this paper.

\begin{definition}
Let $\varphi$ be a formula in the language of set theory.
\begin{enumerate}
\item If $\varphi$ is a statement, we say that $\varphi$ is {\em transferable} if there is a metatheorem stating that, if there exists a Fraenkel--Mostowski model $\mathcal N$ of $\zfa$ satisfying $\varphi$, then there exists a model $\mathcal M$ of $\zf$ that also satisfies $\varphi$.
\item We say that $\varphi$ is a {\em boundable formula} if there is an absolutely definable ordinal $\alpha$ such that, for every $x$, we have that $\varphi(x)$ is equivalent to its relativization $\varphi^{\wp^\alpha(x)}(x)$ (here $\wp^\alpha$ denotes the usual iterated powerset operation, defined recursively by $\wp^0(x)=x$, $\wp^{\xi+1}(x)=\wp(\wp^{\xi}(x))$, and $\wp^\xi(x)=\bigcup_{\beta<\xi}\wp^{\beta}(x)$ for limit $\xi$).
\item A {\em boundable statement} is the existential closure of a boundable formula.
\item We say that $\varphi$ is {\em injectively boundable} if it is a (finite) conjunction of formulas of the form
\begin{equation*}
(\forall y)(\aleph(y)\leq\sigma(x)\Rightarrow\psi(y,x))
\end{equation*}
where $\psi(y,x)$ is a boundable formula and $\sigma(y)$ is a term\footnote{For the purposes of this paper, it is always sufficient to take $\sigma(y)=\aleph_0$.} defined by a boundable formula that depends on $y$ (here $\aleph(y)$ is the Hartogs number of $y$, the least ordinal number that does not inject in $y$).
\item An {\em injectively boundable statement} is the existential closure of an injectively boundable formula.
\end{enumerate}
\end{definition}

The definitions of a boundable formula and statement are from~\cite{jech-sochor}, and all of the other definitions can be found in~\cite{pincus-transfer}. The classical Jech--Sochor theorem~\cite{jech-sochor} (see also~\cite[Theorem 6.1]{jech-choice}) states that all boundable statements are transferable. A generalization of this result was established by Pincus~\cite[Metatheorem 2A6]{pincus-transfer}, who proved that all injectively boundable statements are transferable (note that the class of injectively boundable statements contains all boundable statements and is closed under conjunction, so Pincus's result is stronger that Jech--Sochor's). An even stronger result that will be enough for our purposes is the following.

\begin{theorem}\label{pincus-transfer-theorem}
Any conjunction of a finite number of injectively boundable statements together with any statements among $\op$, $\bpi$, $\dc$, $\cc$, is transferable.
\end{theorem}

\begin{proof}
This is a consequence of \cite[Theorem 4 and note in p. 145]{pincus-add-dep-choice} (see also \cite[p. 547]{pincus-add-dep-choice-to-pi}).
\end{proof}

Pincus's results are even more general (a much more general transfer theorem is stated in \cite[p. 286]{howard-rubin}); here we have stated only what can be expressed in terms of the definitions given so far, which will be enough for our purposes.

Recall that a {\em finiteness class} is a class of sets $\mathscr F$ containing all finite sets, not containing $\omega$, and closed under subsets and bijective images. It is worth noting that all the variations of ``finite'' that we have mentioned here (namely H-finite, Dedekind-finite and C-finite) constitute finiteness classes.

\begin{definition}
We will say that a finiteness class $\mathscr F$ is {\em tame} if there is a boundable formula $\varphi(x)$ such that $\mathscr F=\{x\mid\varphi(x)\}$.
\end{definition}

A glance at the definitions will convince the reader that the classes of H-finite, Dedekind-finite and C-finite sets are all tame (this is also explained, with some more detail, in~\cite[p. 15, third paragraph]{brot-cao-fernandez}).

\begin{theorem}\label{finiteness-classes-transfer}
Let $\mathscr F,\mathscr G$ be tame finiteness classes. Then, both the statement that $\mathscr F=\mathscr G$ and the statement that $\mathscr F\neq\mathscr G$ are 
%for any $F\subseteq n\times n$, the statement
%\begin{equation*}
%\left(\bigwedge_{(i,j)\in F}(\mathscr F_i=\mathscr F_j)\right)\wedge\left(\bigwedge_{(i,j)\notin F}(\mathscr F_i\neq\mathscr F_j)\right)
%\end{equation*}
injectively boundable (and hence transferable).
\end{theorem}

\begin{proof}
Let $\varphi(x),\psi(x)$ be boundable formulas such that $\mathscr F=\{x\big|\varphi(x)\}$ and $\mathscr G=\{x\big|\psi(x)\}$. Note that the class of boundable formulas is closed under Boolean combinations (conjunctions, disjunctions and negations), and that every boundable formula (respectively, statement) is also injectively boundable (respectively, statement). Hence, the statement $\mathscr F\neq\mathscr G$, which is equivalent to $(\exists x)((\varphi(x)\wedge\neg\psi(x))\vee(\neg\varphi(x)\wedge\psi(x)))$, is boundable, hence injectively boundable.

For the remaining statement, recall that the class of Dedekind-finite sets (denoted $\text{D-Fin}$) is the largest finiteness class (this is a consequence of the fact that finiteness classes do not contain $\omega$ and are closed under subsets), and so $\mathscr F,\mathscr G\subseteq\text{D-Fin}$. Hence, the statement that $\mathscr F=\mathscr G$ is equivalent to the statement that for every Dedekind-finite set $x$, $x\in\mathscr F\iff x\in\mathscr G$. It follows immediately from the definition that a set $x$ is Dedekind-finite if and only if $\aleph(x)\leq\omega$. Hence, the statement that $\mathscr F=\mathscr G$ is equivalent to the statement
\begin{equation*}
(\forall x)(\aleph(x)\leq\omega\Rightarrow(\varphi(x)\iff\psi(x))),
\end{equation*}
which is an injectively boundable statement.
\end{proof}

\subsection{The truth-value of $\hind$ in the models}

We now proceed to determine whether $\hind$ holds in each of the four models mentioned at the beginning of the section.

\begin{theorem}\label{prop:firstfraenkel}
$\hind$ does not hold in the First Fraenkel Model $\mathcal N_1$.
\end{theorem}

\begin{proof}
In $\mathcal N_1$ there is an infinite, H-finite set. In fact, the set $A$ of atoms is such a set by~\cite[Proposition 4.2]{brot-cao-fernandez}.
\end{proof}

\begin{theorem}\label{prop:secondfraenkel}
In the Second Fraenkel Model $\mathcal N_2$, $\hind$ holds.
\end{theorem}

\begin{proof}
Take an arbitrary infinite set $X\in\mathcal N_2$ and let us argue that its finite powerset $[X]^{<\omega}$ is Dedekind infinite. If $X$ is well-orderable we are done, so assume that it is not. Take a finite support $F_0:=\bigcup_{i=0}^{n_0}P_i$ for $X$. Now, working in the real world (rather than in $\mathcal N_2$), recursively choose $x_k\in X$ and $n_k\in\mathbb N$ such that $n_k<n_{k+1}$ and $x_k$ is not supported by $F_k=\bigcup_{i=0}^{n_k}P_i$ but it is supported by $F_{k+1}=\bigcup_{i=0}^{n_{k+1}}P_i$ (we can always choose such an $x_k$ because $X$ fails to be well-orderable in $\mathcal N_2$ and so no single finite subset of $A$ can simultaneously support every element of $X$). For each $k<\omega$, the set $Y_k=\{\pi(x_k)\big|\pi\text{ pointwise fixes }F_0\}$ is symmetric (supported by $F_0$) and hence it belongs to $\mathcal N_2$; as $X$ is supported by $F_0$, we have $Y_k\subseteq X$. Furthermore, note that, since $x_k$ is supported by $F_{k+1}$, the value of $\pi(x_k)$ is completely determined by $\pi\upharpoonright F_{k+1}$, whenever $\pi\in G$. There are only finitely many possible values for $\pi\upharpoonright F_{k+1}$ ---in fact, with the requirement that $\pi$ pointwise fixes $F_0$, and given that $F_{k+1}\setminus F_0=\bigcup_{i=n_0+1}^{n_{k+1}}P_i$, there are at most $2^{n_{k+1}-n_0}$ possible values for $\pi\upharpoonright F_{k+1}$. This means that $|Y_k|\leq 2^{n_{k+1}-n_0}$, so $Y_k$ is finite. As each $Y_k$ is supported by $F_0$, we may conclude that the set
\begin{equation*}
f=\left\{\langle k,Y_k\rangle\big|k<\omega\right\}
\end{equation*}
is also supported by $F_0$, and hence $f\in\mathcal N_2$. While $f$ need not be injective, notice that (since the $x_k$ are pairwise distinct and each $Y_k$ is finite) its range must be infinite, hence it is possible to modify $f$ and obtain an injective function $:\omega\longrightarrow[X]^{<\omega}$. Thus, we may conclude that $[X]^{<\omega}$ is a Dedekind-infinite set in $\mathcal N_2$. Therefore $\mathcal N_2\vDash\hind$.
\end{proof}

\begin{theorem}\label{prop:mostowskimodel}
In Mostowski's Linearly Ordered Model $\mathcal N_3$, $\hind$ fails.
\end{theorem}

\begin{proof}
It is known~\cite[Section 4.5]{jech-choice} that $\mathcal N_3\vDash\op$, and this implies that $\mathcal N_3\vDash\kl$; however, $\mathcal N_3\not\vDash\fdf$. Hence, by Proposition~\ref{konigplushindman}, it must be the case that $\mathcal N_3\not\vDash\hind$.
\end{proof}

\begin{theorem}\label{prop:basiccohen}
$\hind$ fails in the Basic Cohen Model $\mathcal M_1$.
\end{theorem}

\begin{proof}
Let $X$ be the set of countably many generic Cohen reals used to construct $\mathcal M_1$. $X$ is Dedekind-finite in $\mathcal M_1$~\cite[Chap. IV, Sec. 9, Theorem 1, p. 138]{cohen-stcontinuum} (using modern notation, the argument in~\cite[Lemma 5.15]{jech-choice} indeed shows this); it is also linearly orderable since $X\subseteq\mathbb R$. Hence, by Proposition~\ref{linearly-ordered-plus-hindman}, $X$ is also H-finite and so $\mathcal M_1\vDash\neg\hind$.
\end{proof}

\subsection{$\hind$ and $\bpi$, $\oep$, $\kw$, $\op$, $\kl$, $\dc$, $\cc$, $\fdf$}

With the information from the previous section under our belt, we are now able to establish which implications between $\hind$ and other choice principles can be proved under $\zf$. Recall that $\hind$ is simply the statement that the class of finite sets coincides with the class of H-finite sets; since both of these are tame finiteness classes, both $\hind$ and $\neg\hind$ are injectively boundable statements by Theorem~\ref{finiteness-classes-transfer}. This fact will be used extensively in what follows.

\begin{theorem}
Under $\zf$, the principles $\dc$, $\cc$ and $\fdf$ imply $\hind$, and none of these implications is reversible.
\end{theorem}

\begin{proof}
Since $\dc\Rightarrow\cc\Rightarrow\fdf\Rightarrow\hind$ (the latter implication due to Corollary~\ref{arrows-ht}), and both $\hind$ and $\neg(\fdf)$ are injectively boundable statements by Theorem~\ref{finiteness-classes-transfer}, Theorem~\ref{pincus-transfer-theorem} implies that it suffices to exhibit a model of $\zfa$ where $\hind$ holds but $\fdf$ fails. By Theorem~\ref{prop:secondfraenkel}, the Second Fraenkel Model $\mathcal N_2$ satisfies this (it is straightforward that the set of atoms $A$ in $\mathcal N_2$ is Dedekind-finite).
\end{proof}

\begin{theorem}\label{ht-and-bpi-co}
In $\zf$, there are no provable implications between $\hind$ and any of $\bpi$, $\oep$, $\kw$, $\op$, $\kl$ (thus $\hind$ is independent of each of these choice principles).
\end{theorem}

\begin{proof}
Cohen's Basic Model $\mathcal M_1$ satisfies $\bpi$ and $\kw$ (this follows from~\cite{halpern-levy}, see also~\cite[p. 146]{howard-rubin}). Since $\bpi\Rightarrow\oep\Rightarrow\op\Rightarrow\kl$, and since $\hind$ fails in $\mathcal M_1$ (by Theorem~\ref{prop:basiccohen}), it follows that neither of $\kw$, $\bpi$, $\oep$, $\op$, or $\kl$, imply $\hind$ in $\zf$.

(For all the choice principles mentioned in the statement of the theorem, except $\kw$, one can obtain an alternative argument by considering Mostowski's Linearly Ordered Model $\mathcal N_3$, where $\hind$ fails (Theorem~\ref{prop:mostowskimodel}). This model satisfies $\bpi$, as proved in~\cite[Section 7.1]{jech-choice}, so it suffices to invoke Theorems~\ref{prop:mostowskimodel} and~\ref{pincus-transfer-theorem} to see that $\bpi$ does not imply $\hind$ over $\zf$ (and hence, neither of $\oep$, $\op$, $\kl$ imply $\hind$ either). This argument, however, does not work for $\kw$ since the latter fails in $\mathcal N_3$, see~\cite[pp. 182--183]{howard-rubin}.)

Conversely, consider the Second Fraenkel Model $\mathcal N_2$. By Theorem~\ref{prop:secondfraenkel}, $\mathcal N_2\vDash\hind$. Since the set of atoms $A$ in $\mathcal N_2$ is Dedekind-finite, it follows from Proposition~\ref{konigplushindman} that $\mathcal N_2\vDash\neg\kl$ (alternatively, one can directly see that the countable sequence of pairs that gives rise to the model, $\langle P_n\big|n<\omega\rangle$, constitutes a countable family of nonempty finite sets without a choice function). Note that the formula $\varphi(x)$ stating ``$x=(T,\leq)$ is an infinite, finitely branching tree without infinite branches'' is boundable (equivalent to its relativization to $\wp^{\omega+1}(x)$) and thus $\neg\kl\equiv(\exists x)(\varphi(x))$ is a boundable statement. Hence, by Theorem~\ref{pincus-transfer-theorem}, $\hind$ does not imply $\kl$ in $\zf$. It follows immediately that $\hind$ does not imply neither of $\op$, $\oep$, $\bpi$, or $\kw$ either.
\end{proof}

\subsection{$\hind$ and various flavours of $\rt$}

So far in this paper, we have only considered the version of Ramsey's theorem dealing with partitions of pairs. Variants of this result in other dimensions have, however, also been considered.

\begin{definition}
Given an $n\in\mathbb N\setminus\{1\}$, the symbol $\rt^n$ will denote the statement that, for every infinite set $X$ and every colouring $c:[X]^n\longrightarrow 2$, there exists an infinite $Y\subseteq X$ such that $[Y]^n$ is $c$-monochromatic.
\end{definition}

It is possible to change the above definition to deal with any finite number of colours, but in any case we wind up with an equivalent statement. More precisely, if we let $\rt^n(k)$ (with $n,k\in\mathbb N\setminus\{1\}$) be the statement that for every infinite set $X$ and every colouring $c:[X]^n\longrightarrow k$, there exists an infinite $Y\subseteq X$ such that $[Y]^n$ is $c$-monochromatic, then it is a theorem of Forster and Truss~\cite[Lemma 2.2]{forster-truss} that, for each given $n$, all of the statements $\rt^n(k)$ as $k$ varies are equivalent under $\zf$. The same two authors also establish~\cite[Theorem 2.3]{forster-truss} that, if $n\leq m$, then $\rt^m\Rightarrow\rt^n$ in $\zf$. We still write $\rt$ without exponent to refer to $\rt^2$, and remind the reader that this statement is Form 17 in~\cite{howard-rubin}. The statement $(\forall n)(\rt^n)$, on the other hand, is referred to as Form 325 in~\cite{howard-rubin}.

In~\cite[Definition 2.1]{brot-cao-fernandez}, a set $X$ is defined to be R$^n$-infinite if for every $c:[X]^n\longrightarrow 2$ there exists an infinite $Y\subseteq X$ such that $[Y]^n$ is $c$-monochromatic (that is, if the $n$-dimensional Ramsey's theorem holds at $X$); and of course $X$ is R$^n$-finite if it is not R$^n$-infinite. Hence, $\rt^n$ is simply the statement that every infinite set is R$^n$-infinite, and statements about the veracity or failure of the principles $\rt^n$ can be thought of as statements about certain finiteness classes being equal. Furthermore, the class of R$^n$-finite sets is tame for every $n\in\mathbb N\setminus\{1\}$, and therefore any of the statements $\rt^n$, $\neg\rt^n$, and their combinations (in conjunction) with $\hind$ and $\neg\hind$ are injectively boundable by Theorem~\ref{finiteness-classes-transfer}.

\begin{theorem}\label{rtindepfromhind}
In $\zf$, there is no provable implication relation between $\hind$ and any of the $\rt^n$ ($n\in\mathbb N\setminus\{1\}$), nor between $\hind$ and Form 325.
\end{theorem}

\begin{proof}
Consider the Second Fraenkel Model $\mathcal N_2$. Theorem~\ref{prop:secondfraenkel} establishes that $\mathcal N_2\vDash\hind$; on the other hand, it is shown in~\cite[Proposition 4.7]{brot-cao-fernandez} that the set of atoms in $\mathcal N_2$ is R$^2$-finite. In particular, $\rt$ fails and {\it a fortiori}, so do each of the $\rt^n$ ($n\geq 3$) as well as Form 325. Since $\hind$ and all of the $\neg\rt^n$ are injectively boundable, it follows from Theorem~\ref{pincus-transfer-theorem} that $\hind$ does not imply any of the Ramsey-theorem-related choice principles in $\zf$.

Conversely, consider the First Fraenkel Model $\mathcal N_1$. We know by Theorem~\ref{prop:firstfraenkel} that $\hind$ fails in $\mathcal N_1$. On the other hand, it is established in~\cite[Proposition 4.1]{brot-cao-fernandez} that the set $A$ of atoms is R$^n$-infinite for all $n\geq 2$; furthermore, by~\cite[Lemma, p. 389]{blass-rt}, every non-well-orderable set from $\mathcal N_1$ contains an infinite subset which is in bijection to a cofinite subset of $A$. Hence, every infinite set in $\mathcal N_1$ contains an infinite subset which is either in bijection with $\omega$, or with a cofinite subset of $A$; since both $\omega$ and $A$ are R$^n$-infinite, it follows that every infinite set is R$^n$-infinite in $\mathcal N_1$. Hence, $\mathcal N_1\vDash(\forall n\in\mathbb N\setminus\{1\})(\rt^n)$; since $\neg\hind$ and all of the $\rt^n$ are injectively boundable, it follows that neither of the $\rt^n$ imply, not even jointly (i.e. as Form 325), the principle $\hind$ in $\zf$\footnote{It is clear that, for each individual $n$, the statement $\rt^n\wedge\neg\hind$ is injectively boundable. The statement of Form 325, however, is at first sight a conjunction of {\it all} of the $\rt^n$ simultaneously. In this case, one needs to verify by hand that the formula $\varphi(x)$ stating that ``for every $n<\omega$ and every $c:[x]^n\longrightarrow 2$, there is an infinite $y\subseteq x$ such that $c$ is constant in $[y]^n$'' is boundable (equivalent to its relativization to $\wp^{\omega+1}(x)$) and hence Form 325, which is equivalent to $(\forall x)(\aleph(x)\leq\omega\Rightarrow\varphi(x))$, is an injectively boundable statement.}.
\end{proof}

\subsection{$\hind$ and other choice principles}

We finish the chapter by briefly mentioning how one can obtain further information, regarding the implication relations (or lack thereof) between $\hind$ and a few other known choice principles. One can obtain plenty of information simply based on Theorems~\ref{prop:secondfraenkel} and~\ref{prop:basiccohen}, which state that $\hind$ holds in the Second Fraenkel Model $\mathcal N_2$ and fails in Cohen's basic model $\mathcal M_1$. For example, there is no implication between $\hind$ and any of Choice from Well-Orderable sets $\mathsf{AC}(\infty,\mathrm{WO})$, Choice from finite sets $\mathsf{AC}(\infty,<\aleph_0)$, and Choice from pairs $\mathsf{AC}(\infty,\leq 2)$, since each of these principles holds in $\mathcal M_1$ (for the first one, see~\cite[Exercise 5.22]{jech-choice}; the remaining two are easily consequences of $\op$) and fails in $\mathcal N_2$ (as witnessed by the partition of the set of atoms in pairs giving rise to the model); of course it is also important to notice that the failure of each of these principles is a boundable statement so we are able to use transfer theorems. It is also worth noting that neither the Hahn--Banach theorem nor the Ultrafilter theorem imply $\hind$, not even jointly, since both of these principles follow from $\bpi$; and $\hind$ does not imply either of these principles either, as witnessed by Solovay's model (which satisfies $\mathsf{DC}$, and therefore also $\hind$).

\section{$\hind$ vs. Form 82}\label{ht-vs-form82}

Recall that we established in Corollary~\ref{arrows-ht} that $\hind$ implies Form 82 under $\zf$. The purpose of this section is to prove that this implication is not reversible. Since Form 82 is equivalent to the statement that every C-finite set is finite, and the class of C-finite sets is tame, it follows from Theorems~\ref{finiteness-classes-transfer} and~\ref{pincus-transfer-theorem} that it suffices to build a Fraenkel--Mostowski permutation model of $\zfa$ satisfying Form 82, but not satisfying $\hind$.

For the construction, we begin with a model of $\zf$ with $|A|=\mathfrak c$, and take a bijection $:\omega^\omega\longrightarrow A$, which we denote by $f\longmapsto a_f$. Recall that $\omega^\omega$ is naturally endowed with a metric space structure, given by declaring $d(f,g)$ to be $0$ if $f=g$ and $\displaystyle{\frac{1}{1+\Delta(f,g)}}$ otherwise, where $\Delta(f,g)=\min\{k<\omega\big|f(k)\neq g(k)\}$. For each $s\in\omega^{­<­­\omega}$ we let $U_s=\{f\in\omega^\omega\big|f\upharpoonright|s|=s\}$; this is an open ball in $\omega^\omega$ with radius $\frac{1}{|s|}$, and the collection $\{U_s\big|s\in\omega^{<\omega}\}$ of all such balls forms a basis for the topology in $\omega^\omega$ induced by the aforementioned metric.

We consider the group $\mathscr G$ consisting of all permutations of $A$ that are induced by isometries of $\omega^\omega$; that is, $\pi\in\mathscr G$ if and only if there exists an isometry $\varphi:\omega^\omega\longrightarrow\omega^\omega$ such that $\pi(a_f)=a_{\varphi(f)}$, in which case we will denote $\pi=\pi_\varphi$. Note that every isometry must map each of the basic open sets $U_s$ to some $U_t$ satisfying $|s|=|t|$. Hence, any such isometry gives rise to, and is entirely determined by, an ``assembly'' of permutations $\langle\varphi_s:\omega\longrightarrow\omega\big|s\in\omega^{<\omega}\rangle$ such that, for each $s\in\omega^{<\omega}$, if $\varphi[U_s]=U_t$ then $\varphi[U_{s\frown n}]=U_{t\frown\varphi_s(n)}$ for all $n<\omega$.

We now proceed to define a filter on $\mathscr G$. For each $n<\omega$ and finite $F\subseteq\omega^\omega$, we define
\begin{equation*}
G_{n,F}=\{\pi_\varphi\big|(\forall f\in F)(\varphi(f)=f)\text{ and }(\forall s\in\omega^n)(\varphi[U_s]=U_s)\}
\end{equation*}
In other words, $G_{n,F}$ consists of all $\pi_\varphi$ where, if $\varphi$ is determined by the assembly of permutations $\langle\varphi_s\big|s\in\omega^{<\omega}\rangle$, then we have for all $k\leq n$ and for all $s\in\omega^k$ that $\varphi_s$ is the identity permutation, and furthermore, for all $k<\omega$ and all $f\in F$ it is the case that $\varphi_{f\upharpoonright k}(f(k))=f(k)$. (Note that $G_{0,\varnothing}=\mathscr G$.) It is easily verified that, for $n,m<\omega$ and finite $E,F\subseteq\omega^\omega$, we have $G_{n,F}\cap G_{m,E}=G_{\max\{n,m\},E\cup F}$, and so the family $\{G_{n,F}\big|n<\omega\text{ and }F\subseteq\omega^\omega\text{ is finite}\}$ generates a filter of subgroups of $\mathscr G$, which we will denote with $\mathscr F$. Furthermore, $\mathscr F$ contains the stabilizer of each atom (the stabilizer of $a_f$ is the subgroup $G_{0,\{f\}}$), and is closed under conjugates (since given $\pi_\varphi\in\mathscr G$, we have $\pi_\varphi^{-1}G_{n,F}\pi_\varphi=G_{n,\varphi^{-1}[F]}$). The filter $\mathscr F$ is therefore a normal filter of subgroups of $\mathscr G$, as defined in~\cite[Chap. 4]{jech-choice}, and so the class $M(A,\mathscr F,\mathscr G)$ of hereditarily symmetric (with respect to this filter and group) sets satisfies $\zfa$.

\begin{lemma}\label{nohindmaninnopalmodel}
In $M(A,\mathscr F,\mathscr G)$, the set $A$ is not H-infinite.
\end{lemma}

\begin{proof}
Suppose, on the contrary, that there exists within $M(A,\mathscr F,\mathscr G)$ a countable injective sequence $\langle A_m\big|n<\omega\rangle$ with each $A_m$ a finite subset of $A$. Let $n<\omega$ and $F\in[A]^{<\omega}$ be such that the enumeration of this sequence is fixed by the elements of $G_{n,F}$. Since the $A_m$ are mutually distinct, there is a $k<\omega$ such that $A_k\not\subseteq F$, so we may pick an $f\in\omega^\omega$ such that $a_f\in A_k\setminus F$, pick $K>\max\{n\}\cup\{\Delta(f,g)\big|g\in F\}$, and let $f'\in U_{f\upharpoonright K}\setminus\{g\in\omega^\omega\big|a_g\in A_k\}$ (this can be done because $A_k$ is finite). Then one can find an isometry $\varphi:\omega^\omega\longrightarrow\omega^\omega$ fixing all $U_s$ for $s\in\omega^n$, fixing each element of $F$, and mapping $f$ to $f'$. Thus, the permutation $\pi_\varphi\in G_{n,F}$ maps $a_f$ to $a_{f'}$ and consequently does not fix $A_k$, contradicting the fact that it fixes the sequence $\langle A_m\big|m<\omega\rangle$.
\end{proof}

In particular, we have that $M(A,\mathscr F,\mathscr G)\vDash\neg\hind$. The remainder of the section is devoted to proving that every infinite set in $M(A,\mathscr F,\mathscr G)$ is C-infinite, and hence this model satisfies Form 82. We begin by considering subsets of $A$.

\begin{lemma}\label{atomsarecinf}
Working in $M(A,\mathscr F,\mathscr G)$, let $X\subseteq A$ be infinite. Then $X$ is C-infinite.
\end{lemma}

\begin{proof}
Begin by noticing that, for each $s\in\omega^{<\omega}$, the set $A_s=\{a_f\big|f\in U_s\}$ belongs to $M(A,\mathscr F,\mathscr G)$, since this set is fixed by all elements of $G_{|s|,\varnothing}$. Furthermore, each permutation in $G_{|s|+1,\varnothing}$ fixes each of the sets $A_{s\frown n}=\{a_f\big|f\in U_{s\frown n}\}$; this implies that the (injective) sequence $\langle A_{s\frown n}\big|n<\omega\rangle$, consisting of subsets of $A_s$, also belongs to $M(A,\mathscr F,\mathscr G)$. This shows that, in $M(A,\mathscr F,\mathscr G)$, each of the sets $A_s$ is C-infinite (in particular, for $s=\varnothing$ we see that $A=A_\varnothing$ is C-infinite). Thus, to prove the lemma it suffices to show that, if $X\subseteq A$ is infinite and belongs to $M(A,\mathscr F,\mathscr G)$, then $A_s\subseteq X$ for some $s\in\omega^{<\omega}$.

To show this, suppose $X\subseteq A$ is infinite and belongs to $M(A,\mathscr F,\mathscr G)$. Let $n<\omega$ and $F\in[A]^{<\omega}$ be such that each element of $G_{n,F}$ fixes $X$. Since $X$ is infinite, we may find an $f\notin F$ such that $a_f\in X$. Pick an $m>\max\{\Delta(f,g)\big|g\in F\}\cup\{n\}$. Notice, then, that for each $f'\in U_{f\upharpoonright m}$ it is possible to find an isometry $\varphi:\omega^\omega\longrightarrow\omega^\omega$ such that $\varphi[U_s]=U_s$ for $s\in\omega^m$, $\varphi(g)=g$ for $g\in F$, and $\varphi(f)=f'$. Since each such isometry $\varphi$ satisfies $\pi_\varphi\in G_{n,F}$, we conclude that $a_{f'}=a_{\varphi(f)}=\pi_\varphi(a_f)\in X$. Therefore $A_{f\upharpoonright m}\subseteq X$.
\end{proof}

We are now in conditions to finish our proof.

\begin{theorem}\label{nopalmodelsatisfies82}
The model $M(A,\mathscr F,\mathscr G)$ satisfies Form 82.
\end{theorem}

\begin{proof}
Working in $M(A,\mathscr F,\mathscr G)$, let $X$ be an arbitrary infinite set. If $X$ is well-orderable, then it is C-infinite, so we may assume without loss of generality that $X$ is not well-orderable. This means that for no $n<\omega$ and for no $F\in[A]^{<\omega}$ can the permutations of $G_{n,F}$ simultaneously fix all elements of $X$. Let $n<\omega$ and $F\in[A]^{<\omega}$ be such that $G_{n,F}$ fixes $X$. We divide the proof in two cases:
\begin{description}
\item[Case 1]Suppose there is an $x\in X$ such that for no $m\geq n$ is it the case that $G_{m,F}$ fixes $x$. In this case, pick an $F'\in[A]^{<\omega}\setminus\{F\}$, with $F\subseteq F'$, of least possible cardinality such that $G_{m,F'}$ fixes $x$ for some $m<\omega$. Fix such $m$, assuming without loss of generality that $m\geq n$. Choose $f\in F'\setminus F$, and choose $k>\max\left(\{m\}\cup\{\Delta(f,g)\big|g\in F'\}\right)$. Consider the set
\begin{equation*}
h=\{(\pi(a_f),\pi(x))\big|\pi\in G_{k,F'\setminus\{f\}}\}.
\end{equation*}
This set is clearly fixed by all permutations in $G_{k,F'\setminus\{f\}}$ and therefore belongs to $M(A,\mathscr F,\mathscr G)$. Furthermore, we claim that $h$ is a function. To see this, suppose that we have two elements $\pi_\varphi,\pi_\psi\in G_{k,F'\setminus\{f\}}$ such that $\pi_\varphi(a_f)=\pi_\psi(a_f)$. This means $\varphi(f)=\psi(f)$, and thus $\varphi^{-1}\psi(f)=f$. Hence, $\varphi^{-1}\psi\in G_{k,F'}\subseteq G_{m,F'}$ and so $\pi_\varphi^{-1}(\pi_\psi(x))=\pi_{\varphi^{-1}\psi}(x)=x$, and therefore $\pi_\varphi(x)=\pi_\psi(x)$. Note also that, since $G_{k,F'\setminus\{f\}}\subseteq G_{n,F}$, the range of $h$ is a subset of $X$.

Now, for each $j\geq k$, notice that the set $A_j=\{a_g\in A\big|\Delta(f,g)=j\}$ belongs to $M(A,\mathscr F,\mathscr G)$ (as it is fixed by all elements in $G_{0,\{f\}}$), and let us consider the restricted function $h\upharpoonright A_j$. The argument breaks into two subcases:
\begin{description}
\item[Subcase 1.A] There exists a $j\geq k$ such that $h\upharpoonright A_j$ is injective. Then the set $h[A_j]$ is in bijection with $A_j$; since the latter is C-infinite by Lemma~\ref{atomsarecinf}, we conclude that so is $h[A_j]$. Since $h[A_m]\subseteq X$, it follows that $X$ is C-infinite. 

\item[Subcase 1.B] For every $j\geq k$, the function $h\upharpoonright A_j$ fails to be injective. We claim that, in this case, each of the functions $h\upharpoonright A_j$ is in fact a constant function, with constant value $x$. To see this fix any $j\geq k$, and pick two permutations $\pi,\rho\in G_{k,F'\setminus\{f\}}$ such that $\pi(a_f),\rho(a_f)\in A_j$, $\pi(a_f)\neq\rho(a_f)$, and $\pi(x)=\rho(x)$. Then $\rho^{-1}\pi(a_f)\in A_j\setminus\{a_f\}$  and $\rho^{-1}\pi(x)=x$. We conclude that there is a $\sigma\in G_{k,F'\setminus\{f\}}$ such that $\sigma(a_f)\in A_j\setminus\{a_f\}$ and $\sigma(x)=x$.

Now take an arbitrary $g\in\omega^\omega$ such that $\Delta(g,f)=j$ (that is to say, take an arbitrary $a_g\in A_j$). Then one can find a $\tau\in G_{k,F'}$ such that $\tau(\sigma(a_f))=a_g$; thus we have $\tau(x)=x$. Notice, then, that $\tau\sigma\in G_{k,F'\setminus\{f\}}$, and so $h(\tau\sigma(a_f))=\tau\sigma(x)$. But $\tau(\sigma(a_f))=a_g$ and $\tau(\sigma(x))=\tau(x)=x$, thus $h(a_g)=x$ and, $a_g$ being arbitrary in $A_j$, we conclude that $h\upharpoonright A_j$ is a constant function with constant value $x$.

Note that, given any $\pi\in G_{k,F'\setminus\{f\}}$, we must have, if $\pi(a_f)=a_g$, that $\Delta(f,g)\geq k$. In other words, $\pi(a_f)\in A_j$ for some $j\geq k$. Then, $x=h(\pi(a_f))=\pi(x)$. The conclusion is that $G_{k,F'\setminus\{f\}}$ fixes $x$, contrary to the assumption about $F'$ being of least possible cardinality.
\end{description}

\item[Case 2]Suppose the assumption from Case 1 does not hold. Then, for every $x\in X$, there is an $m\geq n$ such that $G_{m,F}$ fixes $x$. Since $G_{n,F}$ fixes $X$ setwise, we have an action of $G_{n,F}$ on the set $X$, and therefore $X$ can be written as the disjoint union of orbits. Each of these orbits is of the form $\mathcal O(x)=\{\pi(x)\big|\pi\in G_{n,F}\}$, and is fixed by $G_{n,F}$ ---hence it belongs to $M(A,\mathscr F,\mathscr G)$. We begin by showing that every orbit is, in fact, well-orderable in $M(A,\mathscr F,\mathscr G)$. To see this, take an arbitrary $x\in X$, and let $m\geq n$ be such that $G_{m,F}$ fixes $x$. Then we claim that $G_{m,F}$ fixes every element of $\mathcal O(x)$: for if $\pi(x)\in\mathcal O(x)$, for some $\pi\in G_{n,F}$, and $\rho\in G_{m,F}$ is arbitrary, then a routine calculation shows that $\pi^{-1}\rho\pi\in G_{m,F}$, and therefore $\pi^{-1}\rho\pi(x)=x$, thus $\rho(\pi(x))=\pi(x)$. Since a single element of the filter $\mathscr F$ fixes all elements of $\mathcal O(x)$, a routine argument shows that any well-ordering of $\mathcal O(x)$ ``from the real world'' must also belong to $M(A,\mathscr F,\mathscr G)$ ---fixed by $G_{m,F}$.

In particular, this implies that $X$ cannot be covered with finitely many orbits, since the disjoint union of finitely many well-orderable sets is well-orderable in $\zf$, contrary to our assumption about $X$. Now, the set $\mathscr O=\{\mathcal O(x)\big|x\in X\}\subseteq\wp(X)$ belongs to, and is well-orderable in, $M(A,\mathscr F,\mathscr G)$ ---it and each of its elements being fixed by $G_{n,F}$. In particular, it is Dedekind-infinite, and hence so is $\wp(X)$, which means that $X$ is C-infinite, and we are done.
\end{description}
This finishes the proof.
\end{proof}

\begin{corollary}
In $\zf$, Form 82 does not imply $\hind$.
\end{corollary}

\begin{proof}
By Lemma~\ref{nohindmaninnopalmodel}, the model $M(A,\mathscr F,\mathscr G)$ fails to satisfy $\hind$; this, coupled with Theorems~\ref{nopalmodelsatisfies82},~\ref{finiteness-classes-transfer}, and~\ref{pincus-transfer-theorem}, finishes the proof.
\end{proof}

\section{A weaker Boolean form of $\hind$}\label{section-boolean}

In~\cite{brot-cao-fernandez}, certain variations of H-finite sets are considered. Recall that, given a set $X$, one can equip the set $[X]^{<\omega}$ with the symmetric difference as a binary operation, in order to obtain an Abelian group (in which each element has order 2, hence this is usually called a {\em Boolean group}). In this group, given a family $Y\subseteq X$ one can consider its set of {\em finite sums}, $\fs(Y)=\{y_1+\ldots+y_m\big|m\in\mathbb N\wedge y_1,\ldots,y_m\in Y\}$, where $+$ denotes the Abelian group operation ---the symmetric difference. Note that, for a pairwise disjoint family $Y\subseteq[X]^{<\omega}$, $\fu(Y)=\fs(Y)$ and so one could consider Hindman's finite unions theorem $\hind$ as a special case of an analogous statement on which one obtains a monochromatic $\fs(Y)$ without requiring that $Y$ is a pairwise disjoint family. It turns out that this analogous statement is equivalent to the original one; however, one obtains strictly weaker statements if one starts restricting the number of summands allowed in our finite sums.

\begin{definition}
Given a set $X$, a family $Y\subseteq[X]^{<\omega}$, and an $n\in\mathbb N$, we let $\fs_{\leq n}(Y)=\{F_1\bigtriangleup\cdots\bigtriangleup F_t\big|t\leq n\text{ and }F_1,\ldots,F_t\in Y\}$. For $n,k\in\mathbb N$, we let $\hind_n(k)$ denote the statement that for every infinite set $X$ and every colouring $c:[X]^{<\omega}\longrightarrow k$, there exists an infinite $Y\subseteq[X]^{<\omega}$ such that $\fs_{\leq n}(Y)$ is $c$-monochromatic.
\end{definition}

Hence, $\hind_n(k)$ denotes Hindman's theorem for $k$ colours and finite sums of at most $n$ summands. For a fixed $k$, it follows\footnote{In the two references that follow, what was really proved is the case $k=2$, but it is clear from a cursory reading of the proof that this can be adapted to any $k$.} from~\cite[Theorem 3.2]{brot-cao-fernandez} that $\hind$ is equivalent to $\hind_n(k)$ whenever $n\geq 4$ and $k$ is arbitrary, and $\hind\Rightarrow\hind_3(k)\Rightarrow\hind_2(k)$.

Looking at~\cite[Corollary 4.16]{brot-cao-fernandez} and the subsequent discussion, as well as~\cite[Question 5.1 (1)]{brot-cao-fernandez}, it becomes apparent that $\hind_3(k)$ seems to be very close to $\hind$, and for all we know these two principles could very well turn out to be equivalent over $\zf$.

\begin{question}
Given a fixed $k\geq 2$, is the principle $\hind_3(k)$ equivalent to $\hind$ over $\zf$?
\end{question}

In this section, we study the strength of the principle $\hind_2(k)$, where $k$ is fixed but arbitrary. Although we do not know whether the $\hind_2(k)$ are equivalent for distinct $k$ (of course, it is clear that $\hind_2(k+1)\Rightarrow\hind_2(k)$ for each $k\geq 2$), for the considerations in this paper it does not make a difference which specific $k$ we have fixed, and so we will uniformly study all of the $\hind_2(k)$. Following~\cite[Definition 3.1, Corollary 3.5, Definition 3.6]{brot-cao-fernandez}, we will say that $X$ is H$_2(k)$-infinite if for every $c:[X]^{<\omega}\longrightarrow k$ there exists an infinite $Y\subseteq[X]^{<\omega}$ such that $\fs_{\leq 2}(Y)$ is $c$-monochromatic, and $X$ is H$_2(k)$-finite if it is not H$_2(k)$-infinite. So $\hind_2(k)$ is simply the statement that every infinite set is H$_2(k)$-infinite, and hence both $\hind_2(k)$ and its negation are injectively boundable statements (since the class of H$_2(k)$-finite sets is a tame finiteness class).

Since every R$^2$-infinite set is H$_2$-infinite by~\cite[Theorem 3.8]{brot-cao-fernandez}, we have $\rt\Rightarrow\hind_2(k)$, for any $k\geq 2$. Hence, either of $\hind$ and $\rt$ (and {\it a fortiori} also any $\rt^n$, $n>2$, as well as Form 325) both imply $\hind_2(k)$ for every $k\geq 2$. We now proceed to precisely locate $\hind_2(k)$ among the various choice principles considered in this paper. The information obtained can be seen in the enhanced diagram from Fig.~\ref{fig:enhanced-diagram} (with the main information being the arrows {\it not} shown in the diagram, cf. the discussion in the Introduction).

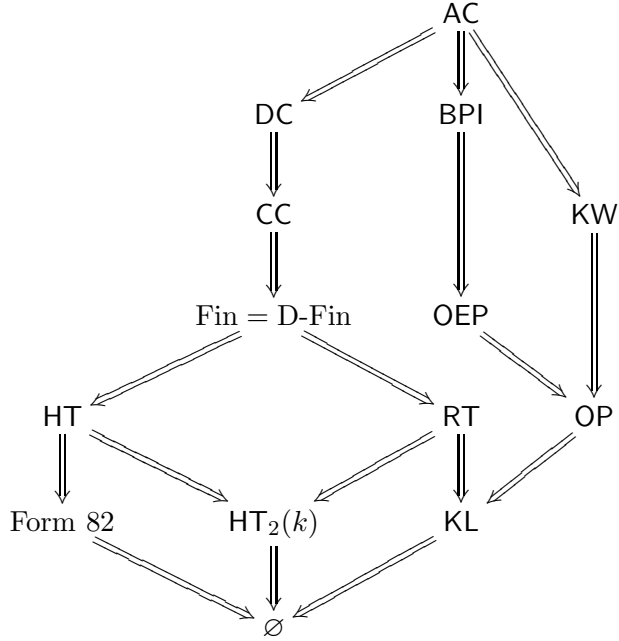
\begin{figure}
\begin{center}
\begin{tabular}{c}
\xymatrix{
 & & \ac\ar@{=>}[dl]\ar@{=>}[ddr]\ar@{=>}[d] & \\
 & \dc\ar@{=>}[d] & \bpi\ar@{=>}[dd] & \\
 & \cc\ar@{=>}[d] & & \kw\ar@{=>}[dd] \\
 & \fdf\ar@{=>}[dl]\ar@{=>}[dr] & \oep\ar@{=>}[dr] & \\
\hind\ar@{=>}[dr]\ar@{=>}[d] & & \rt\ar@{=>}[d]\ar@{=>}[dl] & \op\ar@{=>}[dl] \\
\text{Form 82}\ar@{=>}[dr] & \hind_2(k)\ar@{=>}[d] & \kl\ar@{=>}[dl] & \\
 & \varnothing & & 
}
\end{tabular}
\end{center}
\caption{Enhanced diagram of implication relations, now including $\hind_2(k)$ among the other classical choice principles.}
\label{fig:enhanced-diagram}
\end{figure}

\begin{theorem}\label{the-n-satisfy-ht2}
The models $\mathcal N_1$, $\mathcal N_2$ and $\mathcal N_3$ all satisfy $\hind_2(k)$ for all $k\geq 2$.
\end{theorem}

\begin{proof}
Since $(\hind\vee\rt)\Rightarrow\hind_2(k)$, the statement is immediate from the fact that $\mathcal N_1\vDash\rt$ (by~\cite[Theorem 2]{blass-rt}), $\mathcal N_2\vDash\hind$ (Theorem~\ref{prop:secondfraenkel}), and $\mathcal N_3\vDash\rt$ (this is a theorem of Tachtsis~\cite[Theorem 2.4]{tachtsis-chainsantichains}).
\end{proof}

\begin{corollary}
In $\zf$, $\hind_2(k)$ does not imply $\hind$ or any of the $\rt^n$ (and {\em a fortiori}, it does not imply Form 325 either), for any $k\geq 2$.
\end{corollary}

\begin{proof}
Either of the models $\mathcal N_1$ and $\mathcal N_3$ satisfy $\hind_2(k)$ by Theorem~\ref{the-n-satisfy-ht2}, while neither of them satisfies $\hind$, by Theorems~\ref{prop:firstfraenkel} and~\ref{prop:mostowskimodel}. Thus $\hind_2(k)$ does not imply $\hind$ (in $\zfa$, but also in $\zf$ since $\hind_2(k)\wedge\neg\hind$ is an injectively boundable statement), for any $k\geq 2$.

On the other hand, the model $\mathcal N_2$ satisfies $\hind_2(k)$ but it does not satisfy any of the $\rt^n$ (as argued in the proof of Theorem~\ref{rtindepfromhind}), and so $\hind_2(k)$ does not imply any of the $\rt^n$ (again, originally in $\zfa$ but it follows that this works also in $\zf$ since $\hind_2(k)\wedge\neg\rt$ is an injectively boundable statement), for any $k\geq 2$.
\end{proof}

The choice principle $\hind_2(k)$ is thus strictly weaker than both $\hind$, and all of the $\rt^n$, for any $k\geq 2$. It follows immediately that $\hind_2(k)$ is also strictly weaker than (meaning that it is implied by, with the implication not reversible) all of $\dc$, $\cc$, and $\fdf$. The following theorem will allow us to establish the lack of implication relations between $\hind_2(k)$ and a host of other choice principles.

\begin{theorem}\label{ht2-in-basic-cohen}
In the basic Cohen model $\mathcal M_1$, $\hind_2(2)$ fails (and hence so does $\hind(k)$, for any $k\geq 2$).
\end{theorem}

\begin{proof}
Let us fix some notation: let $\mathbb P$ be the forcing notion where conditions are finite functions $p:F\times n\longrightarrow 2$ for some finite $F\subseteq\omega$, $n<\omega$. Given a permutation $\pi\in\sym(\omega)$, we also use the letter $\pi$ to denote the automorphism $\pi:\mathbb P\longrightarrow\mathbb P$ induced by permuting the columns of $\omega\times\omega$ according to $\pi$. Let $X=\{x_n\big|n<\omega\}$ be the set of Cohen reals (thought of as elements of $2^\omega$) added by $\mathbb P$ ($x_n$ represents the $n$-th column of the generic function $:\omega\times\omega\longrightarrow 2$). (Of course, the enumeration $\langle x_n\big|n<\omega\rangle$ does not belong to $\mathcal M_1$ even though the set $\{x_n\big|n<\omega\}$ does.)

Define the colouring $c:[X]^{<\omega}\longrightarrow 2$ given by $c(F)=1$ if and only if there are distinct $x_n,x_m\in F$ such that $\min(x_n\bigtriangleup x_m)$ is even. We claim that $c$ witnesses that $\hind_2(2)$ fails. To see this, suppose, aiming for a contradiction, that $Y\subseteq[X]^{<\omega}$ is an infinite set such that $\fs_{\leq 2}(Y)$ is $c$-monochromatic. Let $E\subseteq\omega$ be a finite ``support'' for the set $Y$; in other words, let $\mathring{Y}$ be a $\mathbb P$-name for $Y$ such that every permutation $\pi$ fixing each element of $E$, satisfies that $\pi(\mathring{Y})=\mathring{Y}$. Since $Y$ is infinite, there is an $F\in Y$ such that $F\not\subseteq\{x_k\big|k\in E\}$. Let $n<\omega$, $n\notin E$, and $p\in\mathbb P$ be such that $p\Vdash``\mathring{x_n}\in\mathring{F}\in\mathring{Y}"$. Take an $m<\omega$ such that $m\notin\dom(\dom(p))\cup E$, and let $\pi$ be the transposition of $n$ and $m$. Then $\pi(p)\Vdash``\mathring{x_m}\in\pi(\mathring{F})\in\mathring{Y}"$; furthermore, $p$ and $\pi(p)$ are compatible conditions. Thus $p\cup\pi(p)\Vdash``\{\mathring{x_n},\mathring{x­_m}\}=\mathring{F}\bigtriangleup\pi(\mathring{F})\in\fs_{\leq 2}(\mathring Y)"$, and $p\cup\pi(p)$ also forces that $x_n$ agrees with $x_m$ up to $\ran(\dom(p))$. Let $q$ be an extension of $p\cup\pi(p)$ deciding that $\min(\mathring{x_n}\bigtriangleup\mathring{x_m})$ is even; this shows that the colour of $\fs_{\leq 2}(Y)$ is $1$. However, one can run the exact same argument up to the moment where one chooses $q$, at which time let us pick $q$ forcing that $\min(\mathring{x_n}\bigtriangleup\mathring{x_m})$ is odd. This implies that the colour of $\fs_{\leq 2}(Y)$ is also $0$, a contradiction.
\end{proof}

\begin{corollary}
Given a fixed $k\geq 2$, there is no $\zf$-provable implication relation between $\hind_2(k)$ and neither of $\bpi$, $\oep$, $\kw$, $\op$, or $\kl$.
\end{corollary}

\begin{proof}
$\hind_2(k)$ cannot imply neither of the choice principles mentioned in the statement of the theorem, since $\hind$ does not imply them either. On the other hand, all of these choice principles hold in $\mathcal M_1$ (as mentioned in the proof of Theorem~\ref{ht-and-bpi-co}), while $\hind_2(k)$ fails in this model because of Theorem~\ref{ht2-in-basic-cohen}. Hence, neither of the principles mentioned in the statement of the theorem implies $\hind_2(k)$.
\end{proof}

It remains to establish whether there are any implication relations between $\hind_2(k)$ and Form 82. In order to do this, the following theorem will be instrumental.

\begin{theorem}\label{ht2-in-modelo-nopal}
In the model $M(A,\mathscr F,\mathscr G)$ from Sec.~\ref{ht-vs-form82}, $\hind_2(2)$ fails (and hence so does each of the $\hind_2(k)$ for $k>2$).
\end{theorem}

\begin{proof}
Notice that the metric on $A$ induced by the bijection $f\longrightarrow a_f$ from $\omega^\omega$ to $A$ belongs to $M(A,\mathscr F,\mathscr G)$, since it is supported by the empty set; use the letter $d$ to denote that metric. Then, for any two atoms $a_f,a_g$, the number $d(a_f,a_g)$ is the reciprocal of an integer. We can therefore define, within $M(A,\mathscr F,\mathscr G)$, the colouring $c:[A]^{<\omega}\longrightarrow 2$ given by $c(F)=1$ if and only if there are two distinct $a_f,a_g\in F$ such that $1/d(a_f,a_g)$ is even. We claim that the colouring $c$ witnesses the failure of $\hind_2(2)$ in $M(A,\mathscr F,\mathscr G)$.

To see this, suppose that $Y\subseteq[X]^{<\omega}$ is an infinite, pairwise disjoint family such that $\fs_{\leq 2}(Y)$ is monochromatic, and let $E\subseteq\omega^\omega$ be a finite set, and $n<\omega$, such that every $\pi\in G_{n,E}$ fixes $Y$. Since $Y$ is infinite, we can find an $F\in Y$ such that there is $a_f\in F\setminus\{a_h\big|h\in E\}$. Find a $k>\max\{n\}\cup\{\Delta(g,h)\big|g,h\in E\cup\{f\}\cup\{h\big|a_h\in F\}\}$ and a $g\in\omega^\omega\setminus E$ such that $g\upharpoonright k=f\upharpoonright k$, $g\notin E$, $a_g\notin F$, and $f$ differs from $g$ for the first time at the odd number $m>k$. Letting $\pi$ be induced by an isometry of $\omega^\omega$ in such a way that $\pi$ fixes all elements of $(F\setminus\{a_f\})\cup\{a_h\big|h\in E\}$  and $\pi(a_f)=a_g$, we obtain that $\pi(F)=F\cup\{a_g\}\setminus\{a_f\}\in Y$ and therefore $\{a_f,a_g\}=F\bigtriangleup\pi(F)\in\fs_{\leq 2}(Y)$, with $1/d(a_f,a_g)=m+1$ an even number. This shows that the color in which $\fs_{\leq 2}(Y)$ is monochromatic must be $1$; however, running the exact same argument but choosing $g$ in such a way that $m=\Delta(f,g)$ is even shows that this colour must be $0$ as well, a contradiction.
\end{proof}

\begin{corollary}
In $\zf$, there is no provable implication between $\hind_2(k)$ and Form 82, for any $k\geq 2$.
\end{corollary}

\begin{proof}
By Theorem~\ref{ht2-in-modelo-nopal}, the model constructed in Sec.~\ref{ht-vs-form82} satisfies Form 82 together with $\neg\hind_2(k)$. So Form 82 does not imply $\hind_2(k)$, for any $k\geq 2$ (originally in $\zfa$, but these are all statements concerning tame finiteness classes and so they are injectively boundable by Theorem~\ref{finiteness-classes-transfer}).

Conversely, consider the First Fraenkel Model $\mathcal N_1$. We know that $\mathcal N_1\vDash\hind_2(k)$ (cf. Theorem~\ref{the-n-satisfy-ht2}). We claim that Form 82 fails in $\mathcal N_1$. To see this, note that the set of atoms $A$ in $\mathcal N_1$ is amorphous (this is a classical and easy-to-see fact), while at the same time H-finite (by~\cite[Proposition 4.2]{brot-cao-fernandez}). Hence, $A$ must be infinite C-finite, by Proposition~\ref{amorphous-powerset}. In particular, Form 82 fails in $\mathcal N_1$ and so $\hind_2(k)$ does not imply Form 82, for any $k\geq 2$ (once again, in $\zfa$ but the statement is transferable to $\zf$ by Theorem~\ref{finiteness-classes-transfer}).
\end{proof}

We note the curious fact that, in each of the models considered in this section (or in the paper), the principles $\hind_2(k)$ either fail for all $k\geq 2$, or hold for all $k\geq 2$. The argument from Proposition~\ref{colourblind-hindman}, however, does not seem to translate well when the finite sums allowed are limited to a bounded number of summands. Hence, we close the paper with the following natural problem.

\begin{question}
Are all the statements $\hind_2(k)$, for varying $k\geq 2$, equivalent over $\zf$? Does there exist a $\zf$ model satisfying, e.g., $\hind_2(2)$ but not $\hind_2(3)$?
\end{question}

\section*{Acknowledgements}

The author was partially supported by an internal grant from Instituto Polit\'ecnico Nacional ({\em proyecto SIP 20221862}). We are also grateful to the referee for a careful reading of the paper.

\end{document}